\documentclass[CM,GP]{degruyter-crelle}%
\usepackage{amssymb}
\usepackage{amsmath}
\usepackage[matrix,arrow,curve,cmtip]{xy}
\usepackage{mathrsfs}
\usepackage{enumitem}%
\usepackage{amsfonts}%
\usepackage{graphicx}
\providecommand{\U}[1]{\protect\rule{.1in}{.1in}}
\firstpage{}\volume{}\volumeyear{}\copyrightyear{}\doiyear{}\doi{}\received{XXX}\revised{XXX}
\entrymodifiers={+!!<0pt,\fontdimen22\textfont2>}
\setenumerate[1]{leftmargin=2.3em}

\title[Topological Hochschild homology]{Topological Hochschild
homology \\and the Bass trace conjecture}

\author{A. J.}{Berrick}{}{Singapore}
\author{Lars}{Hesselholt}{}{Nagoya}
\contact{Yale-NUS College, 6 College Avenue East, Singapore 138614,
Republic of Singapore and Department of Mathematics, National
University of Singapore, 10 Lower Kent Ridge Road, Singapore 119076,
Republic of Singapore}{jon.berrick@yale-nus.edu.sg, berrick@math.nus.edu.sg}
\contact{Graduate School of Mathematics, Nagoya University, Chikusaku,
  464-8602, Nagoya, Japan and Department of Mathematical 
  Sciences, University of Copenhagen, 2100 Copenhagen \O, Denmark}{larsh@math.nagoya-u.ac.jp, larsh@math.ku.dk}

\researchsupported{Assistance of NUS Research Grant R-146-000-097-112,
Singapore Ministry of Education grant MOE2010-T2-2-125, JSPS
Grant-in-Aid 23340016, and DNRF Niels Bohr Professorship is gratefully acknowledged.}

\numberwithin{equation}{section}
\theoremstyle{plain}
\newtheorem{lettertheorem}{Theorem}

\newtheorem{cor}[equation]{Corollary}
\newtheorem{lemma}[equation]{Lemma}
\newtheorem{prop}[equation]{Proposition}
\theoremstyle{remark}
\newtheorem{example}[equation]{Example}

\begin{document}

\begin{abstract}
We use the methods of topological Hochschild homology to shed new light on
groups satisfying the Bass trace conjecture. Factorization of the
Hattori-Stallings rank map through the B\"{o}kstedt-Hsiang-Madsen cyclotomic
trace map leads to Linnell's restriction on such groups. As a new consequence
of this restriction, we show that the conjecture holds for any group $G$
wherein every subgroup isomorphic to the additive group of rational numbers
has nontrivial and central image in some quotient of $G$.
\end{abstract}


\section*{Introduction}

Let $A$ be a ring, let $K_{0}(A)$ be the projective class group, and let
$\operatorname{HH}_{0}(A)$ be the zeroth Hochschild homology group defined by
the quotient $A/[A,A]$ of $A$ by the additive subgroup $[A,A]$ generated by
all elements of the form $ab-ba$ with $a,b \in A$. The Hattori-Stallings
rank~\cite{hattori,stallings}, also known as the Hattori-Stallings trace, is
the map
$$r \colon K_{0}(A) \longrightarrow\operatorname{HH}_{0}(A)$$
that takes the class of the finitely generated projective right $A$-module $P$
to the image $r(P)$ of the identity endomorphism of $P$ by the trace
homomorphism
$$\smash{ \xymatrix{
{ \operatorname{Hom}_{A}(P,P) } &
{ P \otimes_A \operatorname{Hom}_{A}(P,A) } \ar[l]_-{\sim}
\ar[r]^-{\operatorname{ev}} &
{ A/[A,A]. } \cr
} }$$
Let $G$ be a discrete group, and let $\mathbb{Z}[G]$ be the integral group
ring of $G$. In this case, the group $\operatorname{HH}_{0}(\mathbb{Z}[G])$ is
canonically isomorphic to the free abelian group generated by the set $C(G)$
of conjugacy classes of elements in $G$. Hence, we can write $r(P)$ uniquely
as
$$r(P) = \sum_{[g]\in C(G)}r(P)(g)[g]$$
with $r(P)(g) \in\mathbb{Z}$. It was conjectured by Bass~\cite[Strong
Conjecture]{bass1} that $r(P)(g)$ is always zero, unless $g=1$, and it is this
conjecture that we refer to as the Bass trace conjecture.

In this paper, we consider the following factorization of the
Hattori-Stallings rank through the cyclotomic trace map of
B\"{o}kstedt-Hsiang-Madsen~\cite{bokstedthsiangmadsen}.
$$\smash{ \xymatrix{
{ K_{0}(A) } \ar[r]^-{\operatorname{tr}} &
{ \operatorname{TR}_{0}(A) } \ar[r] &
{ \operatorname{HH}_{0}(A) } \cr
} }$$
The middle group is equipped with operators $F_s$ called the Frobenius
operators, one for every positive integer $s$, and the image of the
cyclotomic trace map is contained in the subgroup fixed by all
Frobenius operators. In the case $A = \mathbb{Z}[G]$, we show that
this leads to the following restriction on the elements $g\in G$ for
which the Hattori-Stallings rank $r(P)(g)$ can be nonzero.

\begin{lettertheorem}\label{linnell}Let $G$ be a group, let $g \in G$
and suppose that $r(P)(g)\neq0$, for some finitely generated
projective $\mathbb{Z}[G]$-module $P$. Then there exists a positive
integer $m=m(g)$ such that for all positive integers $s$, the elements
$g$ and $g^{s^{m}}$ are conjugate.
\end{lettertheorem}

This restriction is not new, but was found earlier by
P.~A.~Linnell~\cite[Lemma~4.1]{linnell} by very different methods,
using work of G.~H.~Cliff~\cite{cliff} on matrices over group
rings. A well-known consequence of the restriction is that if a
nontrivial element $g$ has $r(P)(g)\neq0$, then $g$ lies in a subgroup
isomorphic to the additive group of rationals $\mathbb{Q}$. We present
the following stronger, new consequence.

\begin{lettertheorem}
\label{residuallycentral}Let $G$ be a group such that every subgroup of $G$
that is isomorphic to $\mathbb{Q}$ has nontrivial and central image in some
quotient of $G$. Then the Bass trace conjecture holds for the group $G$.
\end{lettertheorem}

In particular, the Bass trace conjecture holds for any group such that the
intersection of the transfinite lower central series contains no copy of
$\mathbb{Q}$. This is proved in Section~\ref{section3} below, where other
consequences are explored, including the relation between the strong and weak
forms of the Bass trace conjecture; the weak form asserts that $\sum_{[g] \in
C(G) \smallsetminus\{[1]\}} r(P)(g) = 0$.

Below, we use simplicial and cyclic sets and their geometric realization. We
use the notation $X[-]$ to indicate a simplicial or cyclic set with set of
$n$-simplices $X[n]$ and refer the reader to~\cite[Chapter~7]{loday}
and~\cite{drinfeld} for the properties of its geometric realization $|X[-]|$.

\section{The cyclotomic trace map}

\label{section1}

In this section, we recall the cyclotomic trace map and give a
thorough treatment of the factorization of the Hattori-Stallings rank
through said map. We begin with a discussion of the groups
$\operatorname{TR}_q^r(A)$ and of the operators that relate these
groups as the integers $r\geqslant1$ and $q$ vary. The reader is
referred to~\cite[Section~1]{hm4} for further details. The definition
of these groups given in loc.~cit.~uses very large constructions. The
advantage of this is that the various maps that we need all have good 
point-set level models.

Let $\mathscr{C}$ be a (small) exact category in the sense of
Quillen~\cite[\S 2]{quillen}, and let $K(\mathscr{C})$ be the algebraic
$K$-theory spectrum as defined by
Waldhausen~\cite[Section~1.3]{waldhausen}, the definition of which
depends on a choice of null-object $0$ in $\mathscr{C}$. As pointed
out in~\cite[Section~6]{gh}, the spectrum $K(\mathscr{C})$ is a
symmetric spectrum in the sense of~\cite{hoveyshipleysmith}. The
algebraic $K$-groups of $\mathscr{C}$ are defined to be the homotopy
groups
$$K_{q}(\mathscr{C})=[S^{q},K(\mathscr{C})]$$
given by the set of maps in the homotopy category of symmetric spectra
from the suspension spectrum of the sphere $S^{q}$ (which, by abuse of
notation, we again denote by $S^{q}$) to the algebraic $K$-theory
spectrum $K(\mathscr{C})$. For later use, we notice that the zeroth
space of the symmetric spectrum $K(\mathscr{C})$ is defined to be the
geometric realization
$$K(\mathscr{C})_0 = |N(i\mathscr{C})[-]|$$
of the nerve of the subcategory $i\mathscr{C}$ of isomorphisms in
$\mathscr{C}$. There is a canonical map
$$\sigma\colon\pi_{q}(K(\mathscr{C})_{0})\longrightarrow
K_{q}(\mathscr{C})$$
from the homotopy groups of the zeroth space with the chosen
null-object $0$ as the basepoint to the homotopy groups of the
$K$-theory spectrum. It is not an isomorphism, except in trivial
cases, but if $\mathscr{C}$ is split-exact, then it is a
group-completion for $q=0$.

We let $\mathbb{T}$ be the circle group given by the multiplicative
group of complex numbers of modulus $1$. We use the model of the
topological Hochschild $\mathbb{T}$-spectrum $T(\mathscr{C})$ which is
given in~\cite[Section~1.2]{hm4} and which produces a symmetric
orthogonal $\mathbb{T}$-spectrum in the sense
of~\cite{mandellmay,hovey}; it is based on the Dundas-McCarthy
model~\cite[Section~2.1.7]{dundasmccarthy} generalizing B\"{o}kstedt's
original construction~\cite{bokstedt} from rings to exact
categories. The symmetric spectrum structure is used to construct the
cyclotomic trace map, and the orthogonal spectrum structure is used to
define the maps $R_{s}$, $F_{s}$, and $V_{s}$ which we now
explain. First, for every positive integer $r$, the group
$\operatorname{TR}_{q}^{r}(\mathscr{C})$ is defined to be the
equivariant homotopy group
$$\operatorname{TR}_{q}^{r}(\mathscr{C})
= [ S^{q}\wedge(\mathbb{T}/C_{r})_+,\,T(\mathscr{C})]_{\mathbb{T}}$$
given by the set of maps in the homotopy category of symmetric
orthogonal $\mathbb{T}$-spectra from the suspension
$\mathbb{T}$-spectrum of the pointed $\mathbb{T}$-space
$S^q \wedge (\mathbb{T}/C_r)_+$ to $T(\mathscr{C})$. Here 
$C_r \subseteq \mathbb{T}$ denotes the subgroup of order $r$ and the
subscript ``$+$''  indicates the addition of a disjoint basepoint. If
$s$ divides $r$, then there are natural additive maps
$$\begin{aligned}
R_s & \colon \operatorname{TR}_q^r(\mathscr{C}) \longrightarrow
\operatorname{TR}_q^{r/s}(\mathscr{C})  \hskip8mm \text{(restriction)} \cr
F_s & \colon \operatorname{TR}_q^r(\mathscr{C}) \longrightarrow
\operatorname{TR}_q^{r/s}(\mathscr{C}) \hskip8mm \text{(Frobenius)} \cr
V_s & \colon \operatorname{TR}_q^{r/s}(\mathscr{C}) \longrightarrow
\operatorname{TR}_q^r(\mathscr{C}) \hskip8mm \text{(Verschiebung).} \cr
\end{aligned}$$
The map $F_{s}$ is induced by the canonical projection 
$f_s\colon \mathbb{T}/C_{r/s}\rightarrow\mathbb{T}/C_r$, and the map
$V_s$ is induced by the associated transfer map
$v_s \colon S^q \wedge (\mathbb{T}/C_r)_+ \rightarrow S^q \wedge
(\mathbb{T}/C_{r/s})_+$. The definition of the map $R_s$ is more
subtle and uses the so-called cyclotomic structure of
$T(\mathscr{C})$. We briefly recall the definition and refer
to~\cite[Section~2]{hm} for details. In general, if $T$ is a symmetric
orthogonal $\mathbb{T}$-spectrum, then the geometric $C_s$-fixed
point spectrum $\Phi^{C_{s}}T$ is a symmetric orthogonal 
$\mathbb{T}/C_s$-spectrum. In the special case of a suspension
$\mathbb{T}$-spectrum $S^{q}\wedge X$, there is a canonical weak
equivalence of symmetric orthogonal $\mathbb{T}/C_{s}$-spectra
$$i_s \colon S^q \wedge (X^{C_s}) \overset{\sim}{\longrightarrow}
\Phi^{C_s}(S^q \wedge X).$$
The group isomorphism $\rho_s \colon \mathbb{T} \to \mathbb{T}/C_s$
defined by the $s\,$th root gives rise to a change-of-groups functor $\rho
_{s}^{\ast}$ from the category of symmetric orthogonal
$\mathbb{T}/C_{s}$-spectra to the category of symmetric orthogonal
$\mathbb{T}$-spectra. In particular, we obtain the symmetric
orthogonal $\mathbb{T}$-spectrum $\rho_{s}^{\ast}\Phi^{C_{s}}T$. Now,
the cyclotomic structure of the topological Hochschild
$\mathbb{T}$-spectrum $T(\mathscr{C})$ is a collection of compatible
weak equivalences
$$r_s \colon \rho_s^*\Phi^{C_s}T(\mathscr{C}) \longrightarrow
T(\mathscr{C})$$
of symmetric orthogonal $\mathbb{T}$-spectra, one for every positive
integer $s$, and the restriction map
$$R_s \colon \operatorname{TR}_q^r(\mathscr{C}) 
\longrightarrow \operatorname{TR}_q^{r/s}(\mathscr{C})$$
is defined to be the composition
$$\begin{aligned}
{} & \xymatrix{
{ [S^q \wedge (\mathbb{T}/C_r)_+, T(\mathscr{C})]_{\mathbb{T}} }
\ar[r]^-{\Phi^{C_s}} &
{ [\Phi^{C_s}(S^q \wedge (\mathbb{T}/C_r)_+),
\Phi^{C_s}T(\mathscr{C})]_{\mathbb{T}/C_s} } \cr
} \cr
{} & \xymatrix{
{ \phantom{i} } \ar[r]^-{\rho_s^*} &
{ [\rho_s^*\Phi^{C_s}(S^q \wedge (\mathbb{T}/C_r)_+),
\rho_s^*\Phi^{C_s}T(\mathscr{C})]_{\mathbb{T}} } \ar[r]^-{r_{s*}} &
{ [S^q \wedge (\mathbb{T}/C_{r/s})_+, T(\mathscr{C})]_{\mathbb{T}} }
\cr
} \cr
\end{aligned}$$
of the geometric $C_{s}$-fixed point functor, the change-of-groups
isomorphism $\rho_{s}^{\ast}$, and the isomorphism induced by the weak
equivalences $\rho_{s}^{\ast}i_{s}$ and $r_{s}$. We record the
following basic properties of the maps $R_{s}$, $F_{s}$, and $V_{s}$.

\begin{lemma}\label{relations}Let $r$, $s$, and $t$ be positive
integers, and let $d = (s,t)$ be the greatest common divisor. The
following relations \textrm{(i)--(iv)} hold.
\begin{enumerate}
\item[\textrm{(i)}] The maps $R_{1}, F_{1}, V_{1} \colon
\operatorname{TR}_q^r(\mathscr{C}) \to
\operatorname{TR}_q^r(\mathscr{C})$ are the
identity maps.
\item[\textrm{(ii)}] If $st$ divides $r$, then
$$\begin{aligned}
{} & R_s \circ R_t = R_{st}  \colon \operatorname{TR}_q^r(\mathscr{C}) \longrightarrow
\operatorname{TR}_q^{r/st}(\mathscr{C}) \cr
{} & F_s \circ F_t = F_{st} \colon \operatorname{TR}_q^r(\mathscr{C}) \longrightarrow
\operatorname{TR}_q^{r/st}(\mathscr{C}) \cr
{} & V_s \circ V_t = V_{st} \colon \operatorname{TR}_q^{r/st}(\mathscr{C}) \longrightarrow
\operatorname{TR}_q^r(\mathscr{C}) \cr
{} & R_s \circ F_t = F_t \circ R_s \colon \operatorname{TR}_q^r(\mathscr{C}) \longrightarrow
\operatorname{TR}_q^{r/st}(\mathscr{C}) \cr
{} & R_s \circ V_t = V_t \circ R_s \colon \operatorname{TR}_q^{r/t}(\mathscr{C}) \longrightarrow
\operatorname{TR}_q^{r/s}(\mathscr{C}) \cr
\end{aligned}$$
\item[\textrm{(iii)}] If both $s$ and $t$ divide $r$, then
$$F_s \circ V_t = dV_{t/d} \circ F_{s/d} \colon
\operatorname{TR}_q^{r/t}(\mathscr{C})
\longrightarrow \operatorname{TR}_q^{r/s}(\mathscr{C}).$$
\item[\textrm{(iv)}] If both $s$ and $t$ but not $st$ divide $r$, then
$$R_s \circ V_t = 0 \colon \operatorname{TR}_q^{r/t}(\mathscr{C})
\longrightarrow \operatorname{TR}_q^{r/s}(\mathscr{C}).$$
\end{enumerate}
\end{lemma}

\begin{proof}The statements~(i)--(iii) are immediate consequences of
the definitions and of~\cite[Proposition~3.2]{hm}; compare loc.~cit.,
Lemma~3.3. To prove~(iv), we consider the diagram
$$\xymatrix{
{ [ S^q \wedge (\mathbb{T}/C_{r/t})_+, T(\mathscr{C}) ]_{\mathbb{T}} }
\ar[r]^-{\Phi^{C_s}} \ar[d]^-{V_t = v_t^*} &
{ [ \Phi^{C_s}(S^q \wedge (\mathbb{T}/C_{r/t})_+), \Phi^{C_s}T(\mathscr{C})
]_{\mathbb{T}/C_s} } \ar[d]^-{(\Phi^{C_s}v_t)^*} \cr
{ [ S^q \wedge (\mathbb{T}/C_r)_+, T(\mathscr{C}) ]_{\mathbb{T}} }
\ar[r]^-{\Phi^{C_s}} &
{ [ \Phi^{C_s}(S^q \wedge (\mathbb{T}/C_r)_+),
\Phi^{C_s}T(\mathscr{C}) ]_{\mathbb{T}/C_s} } \cr
}$$
which commutes since $\Phi^{C_s}$ is a functor. Moreover, since $s$
was assumed not to divide $r/t$, the fixed set
$(\mathbb{T}/C_{r/t})^{C_s}$ is empty. It follows that the upper
right-hand  term in the diagram is zero, which proves~(iv).
\end{proof}

For completeness, we note that $\operatorname{TR}_{q}^{r}(\mathscr{C})$ is a
module over the ring $\mathbb{W}_{\langle r\rangle}(\mathbb{Z})$ of big Witt
vectors in $\mathbb{Z}$ indexed on the set $\langle r\rangle$ of divisors in
$r$, and that if $s$ divides $r$, then
$$V_s \circ F_s \colon \operatorname{TR}_q^r(\mathscr{C})
\longrightarrow \operatorname{TR}_q^r(\mathscr{C})$$
is equal to multiplication by the element $V_{s}(1)\in\mathbb{W}_{\langle
r\rangle}(\mathbb{Z})$. This fact, however, is not used in this paper. On the
other hand, we need the following result.

\begin{prop}\label{exactsequence}Let $r$ be a positive integer, let
$p$ be a prime number dividing $r$, and write $r = p^{u}d$ with $d$
not divisible by $p$. In this situation, the sequence
$$\xymatrix{
{ \operatorname{TR}_0^d(\mathscr{C}) } \ar[r]^-{V_{p^u}} &
{ \operatorname{TR}_0^r(\mathscr{C}) } \ar[r]^-{R_p} &
{ \operatorname{TR}_0^{r/p}(\mathscr{C}) } \ar[r] &
{ 0 } \cr
}$$
is exact.
\end{prop}

\begin{proof}As a particular case
of~\cite[Proposition~1.1]{angeltveitgerhardthesselholt}
with $q = 0$, $v = 1$, and $\lambda = 0$, we have the exact sequence
$$\xymatrix{
{ \mathbb{H}_0(C_{p^u},\operatorname{TR}^d(\mathscr{C})) } \ar[r]^-{N_{p^u}} &
{ \operatorname{TR}_0^r(\mathscr{C}) } \ar[r]^-{R_p} &
{ \operatorname{TR}_0^{r/p}(\mathscr{C}) } \ar[r] &
{ 0. }
}$$
The left-hand term is the $0$th Borel homology group of
$C_{p^u}$ with coefficients in the symmetric orthogonal
$\mathbb{T}$-spectrum $\operatorname{TR}^d(\mathscr{C}) =
\rho_d^*(T(\mathscr{C})^{C_d})$ and is defined by
$$\mathbb{H}_q(C_{p^u},\operatorname{TR}^d(\mathscr{C}))
= [ S^q \wedge (\mathbb{T}/C_{p^u})_+, E_+ \wedge
\rho_d^*(T(\mathscr{C})^{C_d}) ]_{\mathbb{T}},$$
where $E$ is a free $\mathbb{T}$-CW-complex whose underlying space is
contractible. If $E'$ is another such $\mathbb{T}$-CW-complex, then
there is a $\mathbb{T}$-homotopy equivalence $f \colon E \to E'$, the
$\mathbb{T}$-homotopy class of which is unique, and hence, the Borel
homology group is well-defined up to canonical isomorphism. The
skeleton filtration of $E$ gives rise to a first quadrant spectral
sequence
$$E_{s,t}^2 = H_s(C_{p^u}, \operatorname{TR}_t^d(\mathscr{C})) \Rightarrow
\mathbb{H}_{s+t}(C_{p^u}, \operatorname{TR}^d(\mathscr{C}))$$
from the group homology of $C_{p^u}$ with coefficients in the trivial
$C_{p^u}$-module $\operatorname{TR}_t^d(\mathscr{C})$. In particular, the edge
homomorphism defines an isomorphism
$$\xymatrix{
{ \operatorname{TR}_0^d(\mathscr{C}) = H_0(C_{p^u}, \operatorname{TR}_0^d(\mathscr{C})) }
\ar[r]^-{\sim} &
{ \mathbb{H}_0(C_{p^u},\operatorname{TR}^d(\mathscr{C})). } \cr
}$$
Finally, as was noted in~\cite[Lemma~3.2]{hm}, the composition of the edge
homomorphism and the left-hand map $N_{p^u}$ in the exact sequence at
the top of the proof is equal to the Verschiebung map $V_{p^u}$. This
completes the proof.
\end{proof}

We next recall the cyclotomic trace map of
B\"{o}kstedt-Hsiang-Madsen~\cite{bokstedthsiangmadsen}, but use the
technically better construction given by
Dundas-McCarthy~\cite{dundasmccarthy}. The $\mathbb{T}$-fixed point
spectrum $T(\mathscr{C})^{\mathbb{T}}$ is a symmetric orthogonal
spectrum. We further replace the symmetric spectrum $K(\mathscr{C})$
by its suspension symmetric orthogonal spectrum which we again write
$K(\mathscr{C})$; compare~\cite{hovey}. There is a natural map of 
symmetric orthogonal spectra
$$\tau \colon K(\mathscr{C}) \longrightarrow
T(\mathscr{C})^{\mathbb{T}}$$
given by the canonical map defined in~\cite[p.~14, middle]{hm4}. Now the
cyclotomic trace map is defined to be the natural additive map
$$\operatorname{tr}_{r} \colon K_{q}(\mathscr{C}) \longrightarrow
\operatorname{TR}_{q}^{r}(\mathscr{C})$$
given by the composition
$$[ S^q, K(\mathscr{C}) ] \overset{\tau_{*}}{\longrightarrow} 
[ S^q, T(\mathscr{C})^{\mathbb{T}} ] \overset{\sim}{\longrightarrow} 
[S^q \wedge(\mathbb{T}/\mathbb{T})_+,  \,T(\mathscr{C})]_{\mathbb{T}}
\overset{p_{r}^{*}}{\longrightarrow} 
[ S^q \wedge(\mathbb{T}/C_{r})_+,\,T(\mathscr{C})]_{\mathbb{T}}$$
of the map induced by the map $\tau$, the canonical isomorphism, and
the map induced by the canonical projection $p_r \colon \mathbb{T}/C_r \to
\mathbb{T}/\mathbb{T}$. The next result is essential for our purposes here.

\begin{lemma}\label{imageoftracefixed}If $r$ and $s$ are positive
integers with $s$ dividing $r$, then
$$F_{s} \circ\operatorname{tr}_{r} = \operatorname{tr}_{r/s} = R_{s}
\circ\operatorname{tr}_{r} \colon K_{q}(\mathscr{C}) \longrightarrow
\operatorname{TR}_{q}^{r/s}(\mathscr{C}).$$
\end{lemma}

\begin{proof}The left-hand equality follows immediately from the
definitions since
$$p_r \circ f_s = p_{r/s} \colon \mathbb{T}/C_{r/s} \longrightarrow
\mathbb{T}/\mathbb{T}.$$
Similarly, the right-hand equality follows from the definition of the
restriction map, which we recalled above, since the diagram
$$\xymatrix{
{ S^q \wedge (\mathbb{T}/C_{r/s})_+ }  \ar[rr]^-{p_{r/s}}
\ar[d]^-{i_s} &&
{ S^q \wedge (\mathbb{T}/\mathbb{T})_+ } \ar[d]^-{i_s} \cr
{ \rho_s^*\Phi^{C_s}(S^q \wedge (\mathbb{T}/C_r)_+) }
\ar[rr]^-{\rho_s^*\Phi^{C_s}p_r} &&
{ \rho_s^*\Phi^{C_s}(S^q \wedge (\mathbb{T}/\mathbb{T})_+) } \cr
}$$
commutes.
\end{proof}

We next recall the definition of the $(0,0)\,$th space of
$T(\mathscr{C})$, which we use below. The cyclic nerve of the exact
category $\mathscr{C}$ or, more generally, of any (small) category
$\mathscr{C}$ enriched in pointed sets is defined to be the cyclic
pointed set $N^{\operatorname{cy}}(\mathscr{C})[-]$ with $n$-simplices
$$N^{\operatorname{cy}}(\mathscr{C})[n] =\bigvee_{(P_{0},\dots,P_{n})}
\operatorname{Hom}_{\mathscr{C}}(P_{1},P_{0}) \wedge
\operatorname{Hom}_{\mathscr{C}}(P_{2},P_{1}) \wedge\dots\wedge
\operatorname{Hom}_{\mathscr{C}}(P_{0},P_{n}),$$
the wedge sum ranging over $(n+1)$-tuples of objects of $\mathscr{C}$, and
with cyclic structure maps
$$\begin{aligned}
d_i(f_0 \wedge \dots \wedge f_n) & = \begin{cases}
f_0 \wedge \dots \wedge f_i \circ f_{i+1} \wedge \dots \wedge f_n
& \text{if $0 \leqslant i < n$} \cr
f_n \circ f_0 \wedge \dots \wedge f_{n-1} & \text{if $i = n$ } \cr
\end{cases} \cr
s_i(f_0 \wedge \dots \wedge f_n) & = \begin{cases}
f_0 \wedge \dots \wedge f_i \wedge \operatorname{id}_{P_{i+1}} \wedge
f_{i+1} \wedge \dots \wedge f_n & \text{if $0 \leqslant i < n$} \cr
f_0 \wedge \dots \wedge f_n \wedge \operatorname{id}_{P_0} &
\text{if $i = n$} \cr
\end{cases} \cr
t_n(f_0 \wedge \dots \wedge f_n) & = f_n \wedge f_0 \wedge f_1 \wedge
\dots \wedge f_{n-1}. \cr
\end{aligned}$$
We also define $\mathscr{N}(\mathscr{C})[k]$ be the category of
functors from the category $[k]$ generated by the graph
$0\leftarrow1\leftarrow \dots\leftarrow k$ and define
$\mathscr{N}^{i}(\mathscr{C})[k]$ to be the \emph{full} subcategory
whose objects are all functors $P \colon [k] \to \mathscr{C}$ that
take morphisms in $[k]$ to isomorphisms in $\mathscr{C}$. In
particular,
$$\operatorname{ob}(\mathscr{N}^{i}(\mathscr{C})[k])
= N(i\mathscr{C})[k],$$
but morphisms in $\mathscr{N}^i(\mathscr{C})[k]$ need not be
isomorphisms. The cyclic nerve of $\mathscr{N}^i(\mathscr{C})[k]$ is a
cyclic pointed set, covariantly functorial in $[k]$. Hence, the
collection of these cyclic pointed sets form a simplicial cyclic
pointed set
$N^{\operatorname{cy}}(\mathscr{N}^{i}(\mathscr{C})[-])[-]$, and
$$T(\mathscr{C})_{(0,0)}
= |N^{\operatorname{cy}}(\mathscr{N}^{i}(\mathscr{C})[-])[-]|$$
is defined to be the geometric realization. Again, there is a canonical map
$$\sigma_r \colon \pi_q((T(\mathscr{C})_{(0,0)})^{C_{r}})
\longrightarrow \operatorname{TR}_q^r(\mathscr{C})$$
from the equivariant homotopy groups of the $(0,0)\,$th space to the
equivariant homotopy groups of the symmetric orthogonal
$\mathbb{T}$-spectrum $T(\mathscr{C})$, and this map, too, is
typically not a bijection. We note that, while the definition of the
domain of this map uses only the underlying category enriched in
pointed sets of the exact category $\mathscr{C}$, the definition of
the target uses the exact category structure. We also note that, on
$(0,0)\,$th spaces, the map $\operatorname{tr}_{1}$ is induced by the
map
$$\operatorname{tr}_1 \colon
\operatorname{ob}(\mathscr{N}^i(\mathscr{C})[-]) \longrightarrow
N^{\operatorname{cy}}(\mathscr{N}^{i}(\mathscr{C})[-])[0]$$
that to an object $P$ associates its identity morphism
$\operatorname{id}_P$.

We recall that McCarthy~\cite{mccarthy} has defined the Hochschild
homology of a (small) exact category $\mathscr{C}$. Proceeding as
in~\cite[Section~1.2]{hm4}, this construction gives rise to a
symmetric orthogonal $\mathbb{T}$-spectrum
$\operatorname{HH}(\mathscr{C})$, whose equivariant homotopy groups
we write
$$\operatorname{HH}_q^r(\mathscr{C})
=[S^{q}\wedge(\mathbb{T}/C_r)_+,
\operatorname{HH}(\mathscr{C})]_{\mathbb{T}}.$$
There is no cyclotomic structure on $\operatorname{HH}(\mathscr{C})$,
however, and the groups $\operatorname{HH}_q^r(\mathscr{C})$ with
$r>1$ appear to be of little use. Hence, we consider only the groups
$\smash{ \operatorname{HH}_q(\mathscr{C}) =
\operatorname{HH}_q^1(\mathscr{C}) }$ which are McCarthy's Hochschild
homology groups of $\mathscr{C}$. To define the $(0,0)\,$th space, we
recall that the additive cyclic nerve
$N_{\oplus}^{\operatorname{cy}}(\mathscr{C})[-]$ is the cyclic abelian
groups with $n$-simplices
$$N_{\oplus}^{\operatorname{cy}}(\mathscr{C})[n] 
= \bigoplus_{(P_{0},\dots,P_{n})}
\operatorname{Hom}_{\mathscr{C}}(P_{1},P_{0}) \otimes
\operatorname{Hom}_{\mathscr{C}}(P_{2},P_{1})\otimes\dots\otimes
\operatorname{Hom}_{\mathscr{C}}(P_{0},P_{n})$$
and with the cyclic structure maps defined as for the cyclic nerve,
replacing smash products by tensor products. Now the $(0,0)\,$th space
of $\operatorname{HH}(\mathscr{C})$ is defined to be the geometric
realization
$$\operatorname{HH}(\mathscr{C})_{(0,0)} =
|N_{\oplus}^{\operatorname{cy}}(\mathscr{N}^{i}(\mathscr{C})[-])[-]|$$
of the indicated simplicial cyclic abelian group. There is a canonical
map
$$\sigma_1 \colon
\pi_q(\operatorname{HH}(\mathscr{C})_{(0,0)}) \longrightarrow 
\operatorname{HH}_{q}(\mathscr{C})$$
from the (equivariant) homotopy groups of the $(0,0)\,$th space to the
(equivariant) homotopy groups of the symmetric orthogonal
$\mathbb{T}$-spectrum $\operatorname{HH}(\mathscr{C})$, but in
contrast to $K(\mathscr{C})$ and $T(\mathscr{C})$, it follows
from~\cite[Corollary~3.3.4]{mccarthy} that this map is an isomorphism,
if $\mathscr{C}$ is split-exact.

Finally, there is a canonical map of symmetric orthogonal
$\mathbb{T}$-spectra
$$\ell \colon
T(\mathscr{C})\longrightarrow\operatorname{HH}(\mathscr{C})$$
called the linearization map that on $(0,0)\,$th spaces is induced by the map
$$\ell \colon
N^{\operatorname{cy}}(\mathscr{N}^i(\mathscr{C})[-])[-]\longrightarrow
N_{\oplus}^{\operatorname{cy}}(\mathscr{N}^{i}(\mathscr{C})[-])[-]$$
of simplicial cyclic pointed sets that to $f_{0}\wedge\dots\wedge f_{n}$
associates $f_{0}\otimes\dots\otimes f_{n}$ and replaces wedge sums by direct
sums; compare~\cite[Definition~IV.1.3.5]{dundasgoodwilliemccarthy}. To
summarize the situation, we have the following commutative diagram, which we
need only in the case $q=0$.
\begin{equation}\label{diagram1}
\xymatrix{ 
{ \pi_q(K(\mathscr{C})_0) } \ar[r]^-{\operatorname{tr}_1}
\ar[d]^-{\sigma} &
{ \pi_q(T(\mathscr{C})_{(0,0)}) } \ar[r]^-{\ell} \ar[d]^-{\sigma_1} &
{ \pi_q(\operatorname{HH}(\mathscr{C})_{(0,0)}) } \ar[d]^-{\sigma_1}
\cr
{ K_q(\mathscr{C}) } \ar[r]^-{\operatorname{tr}_1} & 
{ \operatorname{TR}_q^1(\mathscr{C}) } \ar[r]^-{\ell} & 
{ \operatorname{HH}_q(\mathscr{C}) } \cr
}
\end{equation}
The lower right-hand horizontal map is an isomorphism, for
$q\leqslant1$; see loc.~cit. Moreover, if the exact category
$\mathscr{C}$ is split-exact, then the right-hand vertical map is an
isomorphism, for all $q\geqslant0$, and the left-hand vertical map is
a group-completion, for $q=0$.

We now specialize to the case where $\mathscr{C} = \mathscr{P}(A)$ is
the (split-)exact category of finitely generated projective right
modules over a unital associative ring $A$. We define
$$K_{q}(A) = K_{q}(\mathscr{P}(A)), \hskip8mm 
\operatorname{TR}_{q}^{r}(A) =
\operatorname{TR}_{q}^{r}(\mathscr{P}(A)), \hskip8mm 
\operatorname{HH}_{q}(A) = \operatorname{HH}_{q}(\mathscr{P}(A)).$$
The map left-hand vertical map $\sigma$ in~(\ref{diagram1}) induces an
isomorphism of the projective class group of $A$ onto the group
$K_0(A)$. Moreover, we have canonical isomorphisms
$$\xymatrix{
{ A/[A,A] } \ar[r]^-{j} &
{ \pi_0(\operatorname{HH}(\mathscr{P}(A))_{(0,0)}) }
\ar[r]^-{\sigma_1} &
{ \operatorname{HH}_0(\mathscr{P}(A)) } \cr
}$$
which justifies our writing $\operatorname{HH}_0(A)$ for the
right-hand group. Here, the map $j$ takes the class of $a \in A$ to
the class of the map of right $A$-modules $l_{a} \colon A \to A$
defined by $l_{a}(b) = ab$, and its inverse takes the class of 
$f \in\operatorname{Hom}_{A}(P_0,P_0)$ to the trace
$\operatorname{tr}(f) \in A/[A,A]$. We refer the reader
to~\cite[Proposition~2.4.3]{mccarthy} for further discussion.

We record the following result. Henceforth, we refer to the composite map in
the statement as the Hattori-Stallings rank.

\begin{lemma}\label{refineshattoristallings}The following composite
map agrees, under the above identification of its target group with
$A/[A,A]$, with the Hattori-Stallings rank.
$$\xymatrix{
{ K_0(A) } \ar[r]^-{\operatorname{tr}_1} &
{ \operatorname{TR}_0^1(A) } \ar[r]^-{\ell} &
{ \operatorname{HH}_0(A) } \cr
}$$
\end{lemma}

\begin{proof}Comparing definitions, we see that it suffices to show
that the composition of the top horizontal maps in~(\ref{diagram1})
takes the class of $P$ to the class of $\operatorname{id}_P$. It does.
\end{proof}

It follows from~Lemmas~\ref{relations} and~\ref{imageoftracefixed}
that the maps $\operatorname{tr}_{r}$ define a map
$$\operatorname{tr}\colon K_0(A) \longrightarrow
\operatorname{TR}_0(A)=\lim_{r}\operatorname{TR}_0^{r}(A)$$
to the inverse limit with respect to the restriction maps, the limit
indexed by set of the positive integers ordered under
division. Moreover, it follows from Lemma~\ref{relations} that the
collection of Frobenius maps 
$\smash{ F_s \colon \operatorname{TR}_0^r(A) \to \operatorname{TR}_0^{r/s}(A) }$ induces a
Frobenius map
$$F_s \colon \operatorname{TR}_0(A) \longrightarrow
\operatorname{TR}_0(A)$$
and that these maps, in turn, define an action of the multiplicative
monoid of positive integers $\mathbb{N}$ on
$\operatorname{TR}_0(A)$. We write
$\operatorname{TR}_0(A)^{\mathbb{N}}$ for the subgroup of elements
fixed by the $\mathbb{N}$-action. Finally,
Lemma~\ref{imageoftracefixed} shows that the image of the map
$\operatorname{tr}$ is contained in
$\operatorname{TR}_0(A)^{\mathbb{N}} \subseteq
\operatorname{TR}_0(A)$, and Lemma~\ref{refineshattoristallings} shows
that the Hattori-Stallings rank is equal to the composite map
$$\xymatrix{
{ K_0(A) } \ar[r]^-{\operatorname{tr}} &
{ \operatorname{TR}_0(A)^{\mathbb{N}} } \ar[r]^-{i} &
{ \operatorname{TR}_0(A) } \ar[r]^-{\operatorname{pr}_1} &
{ \operatorname{TR}_0^1(A) } \ar[r]^-{\ell} &
{ \operatorname{HH}_0(A), } \cr
}$$
where $i$ is the canonical inclusion. This factorization of the
Hattori-Stallings rank map leads to restrictions on its image, as we
next see.

The cyclotomic structure map
$r_s \colon \rho_s^*\Phi^{C_s}T(\mathscr{C}) \to T(\mathscr{C})$, on
the level of $(0,0)\,$th spaces, is a $\mathbb{T}$-equivariant
homeomorphism, the inverse of which is given by the composition of the
homeomorphisms $\Delta_r$ and $D_r$ defined
in~\cite[Sections~1--2]{bokstedthsiangmadsen}. In particular, the
middle left-hand horizontal map in the following diagram is a
bijection. We use this fact in the case $A=\mathbb{Z}[G]$ and
$\mathscr{C}=\mathscr{P}(A)$, where we consider the commutative
diagram
\begin{equation}\label{diagram2}
\xymatrix{ 
{} & 
{ C(G) } \ar[r]^-{i} \ar[d]^-{j} & 
{ A/[A,A] } \ar[d]^-{j}_-{\sim} \cr
{ \pi_0((T(\mathscr{C})_{(0,0)})^{C_s}) } \ar[r]^-{r_s}_-{\sim}
\ar[d]^-{\sigma_s} & 
{ \pi_0(T(\mathscr{C})_{(0,0)}) } \ar[r]^-{\ell} \ar[d]^-{\sigma_1} &
{ \pi_0(\operatorname{HH}(\mathscr{C})_{(0,0)}) }
\ar[d]^-{\sigma_1}_-{\sim} \cr { \operatorname{TR}_0^s(\mathscr{C}) } \ar[r]^-{R_s}
&
{ \operatorname{TR}_0^1(\mathscr{C}) } \ar[r]^-{\ell}_-{\sim} & 
{ \operatorname{HH}_0(\mathscr{C}) } \cr
}
\end{equation}
in which the top middle vertical map $j$ takes the class of $g\in G$
to the class of the map of right $A$-modules 
$l_{g}\colon A \to A$ defined by $l_{g}(b)=gb$, and the top vertical
map $i$ is induced by the canonical inclusion of $G$ in $A$. We define
the natural set map
$$[-]_s \colon C(G) \longrightarrow \operatorname{TR}_0^s(A)$$
to be the composition of the map $j$, the inverse of the map $r_s$,
and the map $\sigma_s$.

\begin{lemma}\label{bracket}Let $G$ be a group. For every $g \in G$,
and every divisor $t$ of $r$,
$$R_t([g]_r) = [g]_{r/t}, \hskip10mm F_t([g]_r) = [g^t]_{r/t}.$$
\end{lemma}

\begin{proof}The first statement is immediate from the definitions,
and the second statement follows from the definitions
and~\cite[Lemma~3.3, p.~54]{hm}.
\end{proof}

It follows from~(\ref{diagram2}) that the map $[-]_1$ induces an
isomorphism from the free abelian group generated by $C(G)$ onto the
group $\operatorname{TR}_0^1(\mathbb{Z}[G])$. We have the following
generalization.

\begin{prop}\label{TR_0^r}Let $G$ be a group, and let $r$ be a
positive integer. The abelian group
$\operatorname{TR}_0^r(\mathbb{Z}[G])$ is free with a basis consisting
of the elements $V_t([g]_{r/t})$, where $t$ and $[g]$ range over the
divisors of $r$ and the elements of $C(G)$, respectively.
\end{prop}

\begin{proof}The proof is by induction on $r$; the case $r=1$ was
established above. So we fix $r$ and assume that the statement holds
for all proper divisors of $r$. We choose a prime $p$ that divides $r$
and consider the exact sequence
$$\xymatrix{
{ \operatorname{TR}_0^d(\mathbb{Z}[G]) } \ar[r]^-{V_{p^u}} &
{ \operatorname{TR}_0^r(\mathbb{Z}[G]) } \ar[r]^-{R_p} &
{ \operatorname{TR}_0^{r/p}(\mathbb{Z}[G]) } \ar[r] &
{ 0 } \cr
}$$
from Proposition~\ref{exactsequence} with $r = p^ud$ and $d$ not
divisible by $p$. It will suffice to show that the left-hand map
$V_{p^u}$ is injective. But it follows from Lemma~\ref{relations} that
the composite map
$$\xymatrix{
{ \operatorname{TR}_0^d(\mathbb{Z}[G]) } \ar[r]^-{V_{p^u}} &
{ \operatorname{TR}_0^r(\mathbb{Z}[G]) } \ar[r]^-{F_{p^u}} &
{ \operatorname{TR}_0^d(\mathbb{Z}[G]) } \cr
}$$
is equal to multiplication by $p^u$, and since the group
$\operatorname{TR}_0^d(\mathbb{Z}[G])$, inductively, is a free abelian
group, we conclude that $V_{p^u}$ is injective as desired.
\end{proof}

We use Proposition~\ref{TR_0^r} to evalute the group
$\operatorname{TR}_0(\mathbb{Z}[G])$ which, we recall, is defined to
be the inverse limit of the groups
$\operatorname{TR}_0^r(\mathbb{Z}[G])$ with respect to the restriction
maps. To this end, we abuse notation and write $[g]$ for the element
$([g]_r)$ of $\operatorname{TR}_0(\mathbb{Z}[G])$. We also remark that
the maps $\smash{ V_t \colon
\operatorname{TR}_0^{r/t}(\mathbb{Z}[G]) \to
\operatorname{TR}_0^r(\mathbb{Z}[G]) }$ give rise to a map $V_t \colon 
\operatorname{TR}_0(\mathbb{Z}[G]) \to
\operatorname{TR}_0(\mathbb{Z}[G])$. 

\begin{cor}\label{TR_0}Every element $a
\in\operatorname{TR}_0(\mathbb{Z}[G])$ admits a unique expression as
a series
$$a = \sum a_{t,[g]}V_{t}([g])$$
where $t$ and $[g]$ range over the positive integers and the conjugacy classes
of elements in $G$, respectively, and where $a_{t,[g]}$ are integers with the
property that for every positive integer $t$, the coefficient $a_{t,[g]}$ is
nonzero for only finitely many $[g]\in C(G)$.
\end{cor}

\begin{proof}By Proposition~\ref{TR_0^r}, each group
$\operatorname{TR}_0^r(\mathbb{Z}[G])$ in the inverse limit is a free abelian group
with a basis consisting of the elements $V_{t}([g]_{r/t})$, where $t$
and $[g]$ range over the divisors of $r$ and the conjugacy classes of
elements in $G$, respectively. Moreover, if $s$ divides $r$, then
Lemmas~\ref{relations} and~\ref{bracket} show that the restriction map
$$R_s \colon \operatorname{TR}_0^r(\mathbb{Z}[G])
\longrightarrow
\operatorname{TR}_0^{r/s}(\mathbb{Z}[G])$$
is equal to the $\mathbb{Z}$-linear map given by
$$R_s(V_{t}([g]_{r/t})) = \begin{cases}
V_t([g]_{r/st}) & \text{if $st$ divides $r$,} \cr
0 & \text{otherwise.} \cr
\end{cases}$$
We see that the statement follows, since $V_t([g]) \in \operatorname{TR}_0(\mathbb{Z}[G])$
is the unique element that projects to $V_t([g]_{r/t}) \in
\operatorname{TR}_0^r(\mathbb{Z}[G])$, if $t$ divides $r$, and to $0 \in
\operatorname{TR}_0^r(\mathbb{Z}[G])$,
otherwise. 
\end{proof}

\section{Proofs of Theorems~\ref{linnell} and~\ref{residuallycentral}}

\label{section2}

We first use the factorization of the Hattori-Stallings rank map through the
cyclotomic trace map to prove Linnell's
theorem~\cite[Lemma~4.1]{linnell} that we stated as
Theorem~\ref{linnell} in the introduction.

\begin{proof}[Proof of Theorem \ref{linnell}]Let $P$ be a finitely
generated projective $\mathbb{Z}[G]$-module. We may write the
Hattori-Stallings rank of $P$ uniquely as
$$r(P) = a_0[1] + a_1[g_1] + \dots + a_n[g_n],$$
where $[1], [g_1], \dots, [g_n]$ are distinct elements of $C(G)$,
where $a_0$ is an arbitrary integer, and where $a_1, \dots, a_n$ are
nonzero integers. Let $m$ be the minimal exponent of the symmetric
group on $n$ letters (the least common multiple of
$\{1,2,3, \dots ,n\}$). We claim that
$$[g_i] = [g_i^{s^m}]$$
for all positive integers $s$ and for $i = 1, 2, \dots, n$. This will
prove the theorem.
To prove the claim, we consider the image $a = \operatorname{tr}(P) \in
\operatorname{TR}_0(\mathbb{Z}[G])$ by the cyclotomic trace
map. We write the element $a$ uniquely as a series
$$a = \sum a_{t,[g]}V_t([g])$$
as in the statement of Corollary~\ref{TR_0}. By
Lemma~\ref{refineshattoristallings}, we have
$\ell(\operatorname{pr}_1(a)) = r(P)$, which shows that the
coefficient $a_{1,[g]}$ is equal to $a_i$, if $[g] = [g_i]$, and zero,
otherwise. Since $a$ is in the image of the cyclotomic trace map, we
have from Lemma~\ref{imageoftracefixed} that $a = F_s(a)$, for all
positive integers $s$. Let $s = p$ be a prime number. In this case, we
have by Lemma~\ref{relations} that
$$F_p(a) = \sum a_{t,[g]} V_t([g^{p}]) + \sum pa_{t,[g]}V_{t/p}([g]),$$
where, in the left-hand sum, $t$ ranges over all positive integers
prime to $p$, where, in the right-hand sum, $t$ ranges over all positive
integers divisible by $p$, and where, in both sums, $[g]$ ranges over
the conjugacy classes of elements in $G$. We note that the generators
$V_{t/p}([g])$ that appear in the right-hand sum are pairwise
distinct, while the generators $V_{t}([g^{p}])$ that appear in the
left-hand sum need not be distinct. In particular, every generator of
the form $V_{t/p}([g])$ with $t$ divisible by $p^{2}$ appears only in the
right-hand sum and appears with the coefficient
$pa_{t,[g]}$. Therefore, the equation $a=F_p(a)$ implies that, for $t$
divisible by $p^{2}$,
$$a_{t/p,[g]} = pa_{t,[g]}.$$
Equivalently, on writing $t=pu$, we have that, for $u$ divisible by
$p$,
$$a_{u,[g]} = pa_{pu,[g]}.$$
Iterating this equation, we find that, for $u$ divisible by $p$, the
coefficient $a_{u,[g]}$ can be divided by $p$ arbitrarily often, and
therefore, is equal to zero. Since this is true for all prime numbers
$p$, we conclude that the coefficient $a_{t,[g]}$ is equal to zero,
unless $t=1$. Hence,
$$a = a_0[1] + a_1[g_1] + \cdots + a_n[g_n],$$
and by Lemma~\ref{bracket}, the equation $a = F_s(a)$ becomes
$$a_0[1] + a_1[g_1] + \cdots + a_n[g_n]
= a_0[1] + a_1[g_1^s] + \cdots + a_n[g_n^s].$$
It follows, by the uniqueness of this expression, that for every
positive integer $s$, the map
$$\varphi_s \colon C(G) \longrightarrow C(G)$$
that takes $[g]$ to $[g^s]$ restricts to a bijection
$\varphi_s|_S$ of the subset
$S = \{ [g_1], \dots, [g_n] \}$ onto itself. But
$(\varphi_s|_S)^m = \operatorname{id}_S$, by the definition of $m$, so
the claim follows.
\end{proof}

\begin{cor}\label{gQdescribed}Let $g$ be a nontrivial element of the
group $G$ for which there exists a finitely generated projective
$\mathbb{Z}[G]$-module $P$ with $r(P)(g) \neq0$. In this situation,
the element $g$ lies in a subgroup $C$ of $G$ that:
\begin{enumerate}
\item[\textrm{(i)}] is isomorphic to $\mathbb{Q}$,
\item[\textrm{(ii)}] is generated by conjugates of $g$,
\item[\textrm{(iii)}] is contained in a finitely generated subgroup $H$ of $G$,
\item[\textrm{(iv)}] has its elements lying in finitely many $H$-conjugacy
classes, and
\item[\textrm{(v)}] has normal closure the subgroup $[G,C]=[G,g]$ of $G$
generated by all commutators of the form $[h,g]$ with $h \in G$.
\end{enumerate}
\end{cor}

\begin{proof}Parts (i)--(iv) may be found in~\cite[Lemma~4.1]{linnell}
and~\cite[Theorem~3.32]{emmanouil1}; only part~(v) is new. By
Theorem~\ref{linnell}, there exists a positive integer $m$ such that
$$[g]=[g^{s^m}],$$
for every positive integer $s$. In particular, we set $s=2^m-1$ so
that $(2^m-1, s^m-1) = 1$, and choose integers $k$ and $l$ such that
$$k(2^m-1) + l(s^m-1) = 1.$$
Now, there exist elements $x,y \in G$ such that $g^{2^m} = xgx^{-1}$
and $g^{s^m} = ygy^{-1}$, or equivalently, such that $g^{2^m-1} =
[x,g]$ and $g^{s^m-1} = [y,g]$. Therefore,
$$g = [x,g]^k[y,g]^l.$$
This shows that $g \in [G, g]$.  Therefore, the smallest normal
subgroup of $G$ that contains $g$, namely $\langle g \rangle [G,
\langle g \rangle]$, reduces to $[G,g]$. By~(ii) above, $C\subseteq
[G,g]=[G,C]$, which is thereby also the normal closure of $C$.
\end{proof}

\begin{proof}[Proof of Theorem~\ref{residuallycentral}]Let $g\in G$ be
a nontrivial element for which there exists a finitely generated
projective $\mathbb{Z}[G]$-module $P$ with $r(P)(g)$ nontrivial, and
let $C\subseteq G$ be the subgroup specified in
Corollary~\ref{gQdescribed}. We must show that if $N\subseteq G$ is a
normal subgroup and if the image $\bar{C}$ of $C$ in $G/N$ is
contained in the center, then that image is necessarily trivial. Now,
by the assumption on $C$, we have $[G/N,\bar{C}\,]=1$. Hence,
$$C \subseteq [G,g] = [G,C] \subseteq N$$
which shows that the image $\bar{C}$ in $G/N$ is trivial as desired.
\end{proof}

\section{Further consequences}

\label{section3}

In this section, we explore some consequences of the main theorems.

We recall that the lower central series of the group $G$ is defined to be the
descending sequence of subgroups
$$G=G_{1} \supseteq G_{2} \supseteq\cdots\supseteq G_{n}
\supseteq \cdots$$
with $G_{n} = [G, G_{n-1}]$. We write $G_{\omega} = \bigcap_{n\geqslant1}
G_{n}$, and recall that the group $G$ is defined to be residually nilpotent if
$G_{\omega}$ is trivial. Continuing, the transfinite lower central series is
defined as follows. For a successor ordinal $\alpha+1$, we have $G_{\alpha+1}
= [G, G_{\alpha}]$, while, for a limit ordinal $\beta$, we have $G_{\beta} =
\bigcap_{\alpha< \beta} G_{\alpha}$. The series stabilizes once $\alpha$ is
larger than the cardinality of $G$. Its intersection is sometimes called the
relatively perfect radical, and is the maximal subgroup $P \subseteq G$ with
the property that $P=[G,P]$. From the five-term homology exact sequence, it is
the maximal subgroup $P$ such that the canonical projection $G \to G/P$
induces an epimorphism on $H_{2}(-,\mathbb{Z})$ and an isomorphism on
$H_{1}(-,\mathbb{Z})$. For further discussion of this group,
see~\cite{rodriguezscevenels}.

\begin{cor}\label{lowercentralseries} Suppose that the intersection
$P$ of the transfinite lower central series of the group $G$ does not
contain a subgroup isomorphic to $\mathbb{Q}$. Then the group $G$
satisfies the Bass trace conjecture.
\end{cor}

\begin{proof}We show that $G$ satisfies the hypothesis of
Theorem~\ref{residuallycentral}. So let $D \subseteq G$ be a subgroup
isomorphic to $\mathbb{Q}$ and let $N \subseteq G$ be the normal
closure of $D$. We have $[G,N] \subseteq N$, and if also $N \subseteq
[G,N]$, then the subgroup $N$, and hence $D$, would be contained in
the intersection of the lower transfinite central series, contradicting
our hypothesis on $G$. So we conclude that $N$, and hence $D$, have
nontrivial image in $G/[G,N]$. Since this image is central, the
hypothesis of Theorem~\ref{residuallycentral} is fulfilled, as was to
be shown.
\end{proof}

The class of groups with the property of Corollary~\ref{lowercentralseries} of
course includes all hypocentral groups (those where the intersection of the
transfinite lower central series is trivial), and so in particular all
residually nilpotent groups. Previously, it was shown by
Emmanouil~\cite{emmanouil} that residually nilpotent groups satisfy the Bass
trace conjecture, if they are of finite rational homological dimension.
However, as Guido Mislin has kindly pointed out (private communication),
because of Proposition~\ref{locally, virtually}~(i) below, to affirm the
conjecture for all residually nilpotent groups, it suffices to check it for
all finitely generated residually nilpotent groups. Such groups are residually
finitely generated nilpotent, and hence, contain no divisible elements,
because all subgroups of finitely generated nilpotent groups are finitely
generated. Thus, the Bass trace conjecture for residually nilpotent groups
follows from Theorem~\ref{linnell}. For further example of groups whose
satisfaction of the Bass trace conjecture follows from Theorem~\ref{linnell},
we have the following.

\begin{example}
The mapping class group $\Gamma_{g}$ of a smooth orientable closed surface of
genus $g\geqslant3$ is a finitely generated perfect group. Therefore, every
element of $\Gamma_{g}$ lies in the intersection of the transfinite lower
central series. However, $\Gamma_{g}$ contains no copy of $\mathbb{Q}$ because
it is residually finite \cite{grossman}. Hence, the Bass trace conjecture
holds for this group; see also~\cite[Corollary~7.17]{mislin}.
\end{example}

\begin{example}
We let $G=SL(2,\mathbb{Q})$ and consider the elements
\[
g =
\begin{pmatrix}
1 & 1 \cr 0 & 1 \cr
\end{pmatrix}
\hskip8mm h =
\begin{pmatrix}
a & b \cr c & d \cr
\end{pmatrix}
.
\]
The calculation
\[
g^{k} =
\begin{pmatrix}
1 & k\cr 0 & 1\cr
\end{pmatrix}
\hskip8mm hgh^{-1} =
\begin{pmatrix}
1-ac & a^{2} \cr -c^{2} & 1+ac \cr
\end{pmatrix}
\]
shows that $g^{k}$ is conjugate to $g$ if and only if $k$ is a square. In
particular, for every positive integer $r$, we have
\[
[g] = [g^{r^{2}}].
\]
So the element $g$ satisfies the conclusion of Theorem~\ref{linnell} with
$m(g) = 2$. Now, the element $g$ clearly lies in the subgroup
\[
D = \left\{
\begin{pmatrix}
1 & x \cr 0 & 1 \cr
\end{pmatrix}
\mid x \in\mathbb{Q} \right\} ,
\]
which is isomorphic to $\mathbb{Q}$. Nevertheless, we may conclude from
Corollary~\ref{gQdescribed} that the Bass trace conjecture holds for $G$.
Indeed, since $G$ is a linear group, it follows from~\cite[Theorem~9.6]{bass1}
that for every finitely generated subgroup $H \subseteq G$, the divisible
elements in $H$ have finite order. Therefore, the group $G$ does not contain
any subgroup $C$ that satisfies both of the properties~(i) and~(iii) of
Corollary~\ref{gQdescribed}. (For an alternative route, using (v) of
Corollary~\ref{gQdescribed}, we can use the well-known fact that finitely
generated linear groups are residually nilpotent, a situation discussed above.
The affirmation of the Bass trace conjecture for the general case of linear
groups then follows from Proposition~\ref{locally, virtually} below. Of
course, it was also established in Bass' original paper~\cite{bass1}.)
\end{example}

That the groups in the next example satisfy the Bass trace conjecture does not
evidently follow from Theorem~\ref{linnell}, because they contain numerous
subgroups isomorphic to $\mathbb{Q}$.

\begin{example}
Let $F$ be a field of characteristic zero, and let $\Lambda$ be a well-ordered
set. In this situation, the McLain group $M(\Lambda;F)$ is known to be
hypocentral~\cite[Theorem~6.22]{robinson}, and so, by
Corollary~\ref{lowercentralseries} satisfies the Bass trace conjecture.
\end{example}

Next, we have two examples with $m(g) = 1$.

\begin{example}
\label{hallgroup}As pointed out in~\cite{schafer}, the conclusion of
Theorem~\ref{linnell} holds for the following group of P.~Hall~\cite{hall}
with $m = 1$. First, index all prime numbers $p_{n}$ by the set of integers.
Let $C$ be the direct sum of countably many copies of $\mathbb{Q}$, again
indexed by the set of integers. Now let $W$ be the group generated by elements
$\xi,\eta$ subject only to the requirement that all conjugates of $\eta$
commute; write $A$ for the abelian group that they generate. Hall's finitely
generated group $H$ is the semidirect product formed by letting $\xi$ act on
$C$ by shifting the $n$th copy of $\mathbb{Q}$ to the $(n+1)$th copy and by
letting $\eta$ act on all elements of the $n$th copy of $\mathbb{Q}$ by
raising to the $p_{n}$th power. The commutator subgroup $H^{\prime}$ is the
semidirect product of $C$ and $A$, while $H^{\prime\prime}= C$ is a minimal
normal subgroup of $H$. Thus, the subgroup $C$ satisfies all of
Corollary~\ref{gQdescribed}. Nevertheless, the Bass trace conjecture holds for
$H$ since it is soluble and hence amenable~\cite{berrickchatterjimislin}.
\end{example}

\begin{example}
\label{twogeneratedgroup}By~\cite[Corollary~1.2]{osin}, every countable
torsion-free group can be embedded into a (torsion-free) $2$-generated group
$G$ with a unique nontrivial conjugacy class. In particular, every element of
$g \in G$ satisfies the conclusion of Theorem~\ref{linnell} with $m(g) = 1$.
If $C \subseteq G$ is a subgroup isomorphic to $\mathbb{Q}$, then all the
conditions of Corollary~\ref{gQdescribed} hold with respect to $g\in C$
embedded in $G$ (since obviously the conjugates of any nontrivial element
generate the whole group, making $G$ simple). Since $G$ has no quotient with
nontrivial center, the hypothesis of Theorem~\ref{residuallycentral} fails to
hold. It would be interesting to decide whether or not this group $G$
satisfies the Bass trace conjecture.
\end{example}

\begin{example}
It is observed in \cite{berrick} that the Bass trace conjecture holds for all
groups if and only if it holds for all binate groups. Now, binate groups are
perfect (indeed, acyclic, with every element a commutator); and so again every
element lies in the intersection of the transfinite lower central series,
whence Theorem~\ref{residuallycentral} is not in general applicable.
\end{example}

In this context, it is worth recording the following known results on the
class $\mathscr{B}$ of groups for which the Bass trace conjecture holds. The
reader may find the proofs of~(i) and~(ii) in~\cite[Remark 1.5]{emmanouil}
and~\cite[Proposition 3.39]{emmanouil1}, respectively.

\begin{prop}
\label{locally, virtually}The Bass trace conjecture holds for the group $G$,
if either

\begin{enumerate}
\item[\textrm{(i)}] it holds for every finitely generated subgroup $H
\subseteq G$, or

\item[\textrm{(ii)}] it holds for every proper subgroup $H \subseteq G$ of
finite index.
\end{enumerate}
\end{prop}

Thus, the class $\mathscr{B}$ is locally and virtually closed. However, it is
not known whether the class $\mathscr{B}$ is residually closed, that is,
whether a group $G$ with the property that for every $g\in G$, there exists a
group homomorphism $f\colon G\rightarrow G'$ with $G' \in \mathscr{B}$
and $f(g)$ nontrivial is necessarily in $\mathscr{B}$. On the other
hand, it is easy to verify that the class of all groups $G$ in which
every divisible element maps nontrivially to the center of some
quotient of $G$ is a residually closed class.

Finally, we make the following amusing observation concerning the relation
between the strong and weak forms of the Bass trace conjecture. (It is known
that to affirm the weak Bass trace conjecture for all groups, it suffices to
do so for a single group, described in~\cite{berrick}.)

\begin{prop}\label{classprop}The strong form of the Bass trace
conjecture holds for all 
groups if
\begin{enumerate}
\item[\textrm{(i)}] the class $\mathscr{B}$ of groups for which the strong
form of the Bass trace conjecture holds is closed under taking
finitely generated subgroups; and
\item[\textrm{(ii)}] the weak form of the Bass trace conjecture holds for all
finitely generated groups.
\end{enumerate}
\end{prop}

\begin{proof}By Proposition~\ref{locally, virtually}~(i), it suffices
to verify the strong Bass trace conjecture for every finitely
generated group $G$. Now, by~\cite[Theorem~1.1]{osin}, the group $G$
embeds in a finitely generated group $T$ in which any two elements of
the same order are conjugate. But from Corollary~\ref{gQdescribed}
applied to $T$, we know that for every finitely generated projective
$\mathbb{Z}[T]$-module $P$, all nontrivial elements $x \in T$ with
$r(P)(x)$ nonzero have the same, infinite, order. Thus, they lie in a
single conjugacy class in $T$. Therefore, the assumption~(ii) that
the weak Bass trace conjecture holds for $T$, implies that
$T \in \mathscr{B}$. Finally, from the assumption~(i) applied to the
group $T$, we conclude that $G \in \mathscr{B}$.
\end{proof}

\acknowl{This paper was written in part while the second author
visited the National University of Singapore. He would like to thank
the university for its hospitality and support. He also thanks Holger
Reich for helpful discussions.}

\end{document}